\documentclass{amsart}
\usepackage[active]{srcltx}
\usepackage{enumerate}
\usepackage{float}
\usepackage{amssymb}

\usepackage[all]{xy}
\input xy
\xyoption{all}
\CompileMatrices
\usepackage{graphicx}
\input diagrams.sty
\diagramstyle[heads=LaTeX,height=2.5em,scriptlabels]
\newarrow{Into}C--->
\newarrow{Igual}=====
\newarrow{Dashto}....>
\newarrow{Ident}=====
\newarrow{Dotsto}....>
\newarrow{DotsX}{}....{}
\newarrow{Fib}{}---{>>}
\newarrow{Cofib}{>>}---{>}

\newcommand{\cat}{\operatorname{cat}\,}
\newcommand{\indcocat}{\operatorname{indcocat}\,}

\newtheorem{theorem}{Theorem}[section]
\newtheorem{lemma}[theorem]{Lemma}
\newtheorem{corollary}[theorem]{Corollary}
\newtheorem{proposition}[theorem]{Proposition}

\newtheorem{remark}[theorem]{Remark}

\newtheorem{definition}[theorem]{Definition}

\theoremstyle{plain}

\title{Inductive LS cocategory and localisation}

\author{C.~Costoya}
\address[C.~Costoya]{Departamento de Matem{\'a}ticas,
Universidade da Coru{\~n}a, Campus de Elvi{\~n}a, s/n, 15071 - A Coru{\~n}a, Spain.}
\email[C.~Costoya]{cristina.costoya@udc.es}

\author{A.~Viruel}
\address[A.~Viruel]{
Departamento de {\'A}lgebra, Geometr{\'\i}a y Topolog{\'\i}a,
Universidad de M{\'a}\-la\-ga, Apdo correos 59, 29080 M{\'a}laga,
Spain.}
\email[A.~Viruel]{viruel@agt.cie.uma.es}

\begin{document}

\maketitle
\begin{abstract}
In this paper we prove that the inductive cocategory of a nilpotent $CW$-complex of finite type $X$, $\indcocat X$, is bounded above by an expression involving the inductive cocategory of the $p$-localisations of $X$. Our arguments can be dualised to LS category improving previous results by Cornea and Stanley. Finally, we show that the inductive cocategory is generic for $1$-connected $H_0$-spaces of finite type.
\end{abstract}

%%%%%%%%%%%%%%%%%%%%%%%%%%%%%%%%%%%%%%%%%%%%%%%%%%%%%%%%
% 1.- Introduction
%%%%%%%%%%%%%%%%%%%%%%%%%%%%%%%%%%%%%%%%%%%%%%%%%%%%%%%%

\section{Introduction}

In this note, we study
the inductive cocategory of a pointed $CW$-complex $X$,  $\indcocat X$ \cite{Ga1}. We discuss the problem of finding an upper bound for $\indcocat X$ when  $\indcocat X_{(p)}$,  the inductive cocategory of the $p$-loca\-li\-sa\-tion of $X$, is known for every prime $p$. In this work, $X$ is thought of as an arithmetic square, that is, $X$ is the homotopy pullback of the following arrows
\begin{equation*}
\xymatrix{
X_{(0)} \ar[rr]^-{({{\{e_{(p)}\}}_p})_{(0)}}  &&( \Pi_p X_{(p)})_{(0)} && \Pi_p X_{(p)} \ar[ll]_-{e_{(0)}}    }
\end{equation*}
where
$\xymatrix{{e_{(p)}} :X \ar[r]&  X_{(p)}}$ denotes the $p$-localisation, $ \Pi_p X_{(p)} $ the local expansion of $X$, the arrow  pointing to the right
denotes the rationalisation of  the ${\{e_{(p)}\}}_p$ and the arrow pointing to the left is the rationalisation of the local expansion of $X$. This is always possible when $X$ is nilpotent  \cite[Theorem 4.1, Example 2.6]{DDK}.

This problem has been first raised about $\cat X$, the Lusternik-Schnirel\-mann category of $X$.
Namely, Toomer \cite{Too} motivated by his Theorem 4 asserting that for simply connected $CW$-complexes $\cat X_{(p)} \leq \cat X$, raises the question (suggested by P. Hilton) of when $\cat X$ is the supremum of $\cat X_{(p)}$ over all the primes. Stanley \cite{Sta} improving a result of Cornea \cite{Cor} shows that
\begin{equation}\label{2m}
\cat  X \leq 2\operatorname{sup} _p  \{  \cat X_{(p)} \}
\end{equation}
for finite type $1$-connected $CW$-complexes.
This inequality is not known to be sharp for spaces verifying $\operatorname{sup}_p \{\cat X_{(p)} \}=m >1 $. For $m=1$,
 Roitberg \cite{Roi1}  constructs an infinite space  such that $X_{(p)}$ is a co-$H$-space for every $p$ (thus $ \cat X_{(p)} = 1 $) though $\cat X = 2$.
  %No example of  $\operatorname{sup}_p \{\cat X_{(p)} \}=m >1$ is known to verify $\cat X = 2m.$
  In his (failed) attempt to show that (\ref{2m}) is sharp for any $m>1$, Roitberg points out relevant examples of infinite spaces verifying $\cat X = m+1$.  We
 believe that the stronger inequality
 \begin{equation}\label{m}
 \cat  X \leq \operatorname{sup} _p  \{  \cat X_{(p)} \} + 1
\end{equation}
holds for almost every ``common'' space.
Moreover, if we add the assumption on $X$ to be finite, no examples at all are known to verify that (\ref{m}) is sharp.
Indeed, the first author  \cite{Cos3} proves  that if $ \operatorname{sup}_p \{\cat X_{(p)} \}=1 $, then necessarily $\cat X=1$, so Roitberg example above cannot be reproduced for finite spaces.

In this paper we find an upper bound,  for  the inductive cocategory of nilpotent finite type $CW$-complexes, of the same flavor as (\ref{2m}).
\begin{theorem}\label{th-1}
Let $X$ be a finite type nilpotent $CW$-complex, then
\begin{equation*}
\indcocat X\leq \operatorname{sup}_p\{\indcocat X_p\}+\indcocat X_0.
\end{equation*}
\end{theorem}

The proof of this theorem is carried out in Section \ref{Sec3}. Following the same lines and using  \cite[Theorem 4.2]{Ga3} it is immediate to prove  (details are omitted) the following refinement of (\ref{2m}).

\begin{theorem}\label{th-2} Let $X$ be a finite type $1$-connected $CW$-complex, then
$$\cat X \leq  \operatorname{sup}_p\{\cat X_p\}+\cat X_0.$$
\end{theorem}

We recall here that $X$ is an $H_0$-space if its rationalisation is an $H$-space (thus $\indcocat X_{(0)} = 1$).
Dually, $X$ is a co-$H_0$-space  if  its rationalisation is a co-$H$-space (thus $\cat X_{(0)} = 1$). From Theorems \ref{th-1} and \ref{th-2} we deduce:

\begin{corollary}\label{mm} Let $X$ be finite-type $CW$-complex. Then
\begin{enumerate}[i)]
\item If $X$ is a nilpotent  $H_0$-space. Then
$$
\indcocat  X \leq \operatorname{sup}_p  \{  \indcocat X_{(p)} \} + 1.
$$

\item If $X$ is a 1-connected co-$H_0$-space. Then
$$ \cat  X \leq \operatorname{sup}_p  \{  \cat X_{(p)} \} + 1.$$

\end{enumerate}
\end{corollary}

Since finite $1$-connected $H_0$-spaces verify also (\ref{m})  (see \cite{Cos2}), one might expect that
finite $1$-connected co-$H_0$-spaces could satisfy Corollary \ref{mm}.{\it i)}.
The main obstruction to prove such a result is that we do not know the minimal models for Ganea cofibrations (see Remark  \ref{cofiber}).
The upper bound given in Theorem \ref{th-1} is sharp for infinite spaces the same as happened with \eqref{2m}. In \cite{Pan} Pan, inspired in the Roitberg's example above-mentioned, shows an infinite space $X$ such that all its $p$-localisations are $H$-spaces (thus $\indcocat X_{(p)}=1$) although $X$ is not ($\indcocat X=2$).
If we add the assumption on $X$ to be finite, Zabrodsky \cite{Zab1} proved for finite spaces that if $X_{(p)}$ is an $H$-space for every prime $p$, then $X$ is itself an $H$-space. Hence the inequality of Theorem \ref{th-1} is not sharp for $m= \operatorname{sup}_p \indcocat =1.$

A problem somewhat related to the previous one was posed by McGibbon on his survey \cite{McG}.  We recall that finite-type spaces not necessarily homotopic verifying $X_{(p)} \simeq Y_{(p)}$  for every prime $p$ are told to be in the same Mislin Genus.  Hence Theorem \ref{th-1} can be reread in terms of the Mislin Genus:
\begin{corollary}\label{cor1}
Let $X$ and $Y$ be nilpotent finite type spaces in the same Mislin Genus. Then $\mid \indcocat Y- \indcocat X \mid \leq r $ where $r=\indcocat Y_{(0)}=\indcocat X_{(0)}$.
\end{corollary}

We recall that the inductive cocategory  of a product is the supremum of the inductive cocategory of their factors \cite[Theorem 2.7]{Ga1}. Using this result, we can prove:
\begin{theorem}\label{Th-2}
Let $X$ be a finite-type $1$-connected $H_0$-space (with a finite number of homotopy groups or a finite number of homology groups). If $Y$ is in the Mislin genus of $X$,  then $\indcocat X=\indcocat Y$.
\end{theorem}

\begin{proof} According to Zabrodsky's non cancellation result \cite[Theorem 3.5]{Zab1}  (see also \cite{Zab2})  there exists  a product of (odd dimensional) spheres $\mathbb S$,  such that $X \times  \mathbb S  = Y \times \mathbb S$.
Then $$\indcocat (X \times  \mathbb S) = \operatorname{max}\{ \indcocat X, \ \indcocat \mathbb S\}.$$
Since the inductive cocategory of an odd dimensional sphere is either $1$ (for $S^1$, $S^3$, $S^7$) or $2$ (otherwise) \cite[p.\ 28]{Ga4}, if this maximum is greater or equal to $3$, then  $\indcocat X$ and $\indcocat Y$ coincide. If the maximum is less than or equal to $2$, we use \cite[Corollary 2.9]{Zab1} asserting  $\indcocat X = 1$ if and only if $\indcocat Y =1$ to conclude.
\end{proof}

\noindent\textbf{Notation.} In what follows, all spaces we consider have the homotopy type of a $CW$-complex. By abuse of notation, we identify the maps between spaces and its homotopy classes. Finally, we understand $p$-localisation as in \cite{HMR}.

%%%%%%%%%%%%%%%%%%%%%%%%%%%%%%%%%%
% 2.- Inductive cocat
%%%%%%%%%%%%%%%%%%%%%%%%%%%%%%%%%%

\section{Inductive cocategory}\label{Sec2}

The cocategory is  presented by Ganea  \cite{Ga1}, \cite{Ga1a}, \cite{Ga1b}, \cite{Ga2} as the dual in an Eckmann-Hilton sense of the Lusternik Schnirelmann category. In some sense, it is a more natural algebraic invariant. The same way the nilpotency of a group $G$ measures how far is $G$ from being abelian, the cocategory of $X$ measures how far $X$ is from being an $H$-space. For example, the cocategory of the classifying space of a discrete group $G$ coincides with the nilpotency of $G$ \cite[Theorem 2.7]{Ga1b}, and the cocategory of a space is an upper bound of the  nilpotency of the space  \cite[Definition 2.2, Theorem 2.4]{Ga1b}.

Several notions of cocategory, other than Ganea's,  exist in literature  (see e.g.\ \cite{A}, \cite{Hop}, \cite{Hov}) though they are not known to be equivalent. In particular, the second author of this paper  together with A.\ Murillo \cite{MV}  have introduced the dual to the Whitehead approach of the LS category which has the advantage that is far more computable than the previous mentioned.

To our purpose of $p$-localising we need the functoriality of the cofiber-fiber construction of the inductive cocategory of Ganea \cite[Definition 6.1, Remark 6.2]{Ga2}:

\begin{definition}\label{def2} The
$n$-th Ganea cofibration of
$X$,
$X \overset{q_n}\rTo G_nX\rTo C_nX $, is defined inductively as follows:
$q_0$ is the cofibration $X\overset{q_0}\rTo CX\rTo\Sigma X$.
Next consider $F$ the
fibre of
$G_{n-1}X\rTo C_{n-1}X$ and factor
$q_{n-1}$ through $F$ to get a  map $X\rTo F$. The associated cofibration to
that map is by definition the $n$-th Ganea
cofibration of $X$. This can
be better viewed in the following diagram:
\begin{diagram}[size=1cm,labelstyle=\scriptstyle]
&&&X&&&\\
&&\ldTo(3,2)^{{q_n}}\ldTo(1,2)^{{q_{_{n-1}}}}&&\rdTo(1,2)_{{{q_{_{1}}}}}\rdTo(3,
2)_{{q_0}}&&\\
G_nX&\rTo^{\iota_n}&G_{n-1}X&\rDots&G_1X&\rTo^{\iota_1}&CX\\
\dTo&&\dTo&&\dTo&&\dTo\\
C_nX&&C_{n-1}X&&C_1X&&\Sigma X\\
\end{diagram}
Then, $\indcocat X$ is the least integer $n$ for which $q_n$ has a homotopy
retraction.
\end{definition}
\begin{remark}\label{cofiber}The reason for the cocategory  not being as popular among specialists as the LS category relies in the fact that the homotopy type of the cofibers $C_n(X)$ in the construction above are not known, whether in the dual fiber-cofiber construction, we do know that $F_n (X) = \overset{n+1}{\ast} \Omega X$  (see \cite[Remark 3.5]{Ga2}).
\end{remark}

In \cite[Theorem 7]{Too} Toomer explains how the $p$-localisation of $1$-connected spaces of finite type behaves nicely with respect the cofiber-fiber construction.  Considering the generalization of Whitehead Theorem in \cite{Dror}, the result can be extended to nilpotent spaces:

\begin{theorem}\label{toomer} Let $X$ be a nilpotent space of finite type. Then $(G_n X)_{(p)}\simeq G_n(X_{(p)})$, and as a consequence if  $\indcocat X \leq n$, then $\indcocat X_{(p)} \leq n$.
\end{theorem}

The idea of the proof is to consider the following diagram
$$ \xymatrix{&X_{(p)} \ar[rd]^{q_n(X_{(p)} )} \ar[ld]_{{(q_n)_{(p)}}}& \\
(G_{n}X)_{(p)}  \ar[rr] ^{l_p}_{\simeq}&&  G_n(X_{(p)})  }$$
where ${{(q_n)_{(p)}}}$ is the $p$-localisation of $q_n$, the $n$-th Ganea cofibration, $q_n(X_{(p)} )$ is the Ganea cofibration for $X_{(p)}$, and $l_{p}$ exists by the universal property of $p$-localisation. In order to show that $l_{p}$ is an homotopy equivalence, we have to use \cite[Theorem 3.1, Example 4.3]{Dror} where conditions are given for an homological equivalence to be an homotopical equivalence. Then, it suffices to consider a retraction of $q_n$ and $p$-localise it to conclude. Conversely:

\begin{lemma}\label{lemilla}
Let $X$ be a nilpotent finite type space. Then the $n$-th Ganea cofibration of $X$ admits a homotopy retraction if and only if for every prime $p$ there exists $r_p$ a homotopy retraction of the  $n$-th Ganea cofibration of $X_{(p)}$ such that $(r_p)_0=r_0.$
\end{lemma}

\begin{proof} Consider  arithmetic squares as we mentioned at the beginning of this paper:
\begin{equation*}
\resizebox{1\textwidth}{!}{
\xymatrix{
X_{(0)} \ar[rr]^-{({{\{e_{(p)}\}}_p})_{(0)}}  \ar[d]_{{(q_n)}_{(0)}}&&( \Pi_p X_{(p)})_{(0)} \ar[d]^{{\big(\Pi_p {(q_n)}_{(p)}\big)}_{(0)}}& \Pi_p X_{(p)} \ar[d]^{{\Pi_p {(q_n)}_{(p)}}} \ar[l]_-{e_{(0)}}
&&& X \ar[d]_{q_n} \\
(G_nX)_{(0)}  \ar[d]_{r_0}  \ar[rr]^-{({{\{e_{(p)}\}}_p})_{(0)}}  &&( \Pi_p (G_nX)_{(p)} )_{(0)}  \ar[d]^{\big(\Pi_p r_p \big)_{(0)}} & \Pi_p (G_nX)_{(p)}
 \ar[d]^{\Pi_p r_p }   \ar[l]_-{e_{(0)}}
  & \ar@{=>}[r]^{\text{homotopy}}_{\text{pullback}}&&  G_nX  \ar@{.>}[d]\\
X_{(0)} \ar[rr]^-{({{\{e_{(p)}\}}_p})_{(0)}}  &&( \Pi_p X_{(p)})_{(0)} & \Pi_p X_{(p)} \ar[l]_-{e_{(0)}}   &&& X  }}
\end{equation*}
By a standard argument, the induced (dotted) map is a retraction of $q_n$ up to a self-homotopy equivalence of $X$. In fact, if we take their composition and we $p$-localise, we obtain $Id_{X_{(p)}}$ for every prime $p$. Since $X$ is nilpotent of finite type, we have that this composition is a self-homotopy equivalence of $X$. Composing with the inverse of this equivalence provides the desired retraction.
\end{proof}

%%%%%%%%%%%%%%%%%%%%%%%%%%%%%%%%%%%%%%%%%%%%%%%%%%%%%%%%
% 3.- Homotopy retractions of the Ganea cofibration
%%%%%%%%%%%%%%%%%%%%%%%%%%%%%%%%%%%%%%%%%%%%%%%%%%%%%%%%

\section{Homotopy retractions of the Ganea cofibration}\label{Sec3}
The aim of this section is to provide a way of comparing two different retractions of the Ganea cofibration without using the homotopy type of the cofiber $C_n(X)$ (see Remark \ref{cofiber}), as Ganea did in  \cite[Section 4]{Ga3}. In fact, using this result of Ganea, it is immediate to prove the refinement
of (1) in Theorem \ref{th-2}.

We use the monodromy action associated to a fibration to compare retractions. We have included at the end of the paper an appendix that describes results on actions that we need in next the proposition.

\begin{proposition}\label{prop:Ganea}
Let $A\rTo^{a} X \rTo^{j} C$  be a cofibration, $F_j\rTo^{\iota} X $ be the inclusion of the homotopy fiber of $ X \rTo^{j} C$, and $F \rTo^{i} E \rTo^p B$ be a fibration.
Then,
for  any $X\rTo^g F$ and for any retraction $A\rTo^{a} X\rTo^r A$, the following holds:
$$i  g= i  g  a  r \ \Leftrightarrow \ g  \iota=  g  a   r  \iota$$
\end{proposition}
\begin{proof}
For any space $Y$, let $+$ denote the group operation in $[Y,\Omega B]$  and let $\vdash$ denote the $[Y,\Omega B]$-action on $[Y,F]$ (see Apendix \ref{appendix}).
Since $ i  g= i  g  a  r$,  by Theorem \ref{thoperation} there exists an element $\epsilon\in [X,\Omega B]$ such that $g=( \epsilon\vdash g  a  r)$. Moreover, $r  a= 1_A$, and therefore $$g  a=( \epsilon\vdash g  a  r) a = (\epsilon  a\vdash g  a  r a) =  (\epsilon a\vdash g  a) .$$
Using again that $ra= 1_A$, we deduce that $(\epsilon - \epsilon ar)a=0$, hence there exits an element $\psi\in [C,\Omega B]$ such that $\psi j=\epsilon - \epsilon ar$, that is, $\epsilon=\psi j + \epsilon ar$.

Therefore $g=\big((\psi j+ \epsilon ar)\vdash gar \big)$ and (notice that $j\iota =0$)
\begin{align*}
g \iota & =\big((\psi j + \epsilon ar)\vdash gar\big)\iota\\
& =\big( (\psi j + \epsilon ar) \iota \vdash gar\iota \big)\\
& =\big( (\psi j \iota  + \epsilon ar \iota)\vdash gar\iota \big)\\
& = \big(\epsilon ar\iota \vdash gar\iota \big) \\
& =(\epsilon a \vdash ga)r\iota \\
& = gar\iota
\end{align*}
\end{proof}

We can now prove the main result in this section. Let $\iota ^n_m$ denote the composition $\iota_{m+1}\iota_{m+2}\cdots\iota_n\colon G_nX\rightarrow G_mX$ where $\iota_k$ is as in Definition \ref{def2}. Then,

\begin{theorem}\label{th:sections}
Let $X$ be a space such that $ \indcocat X \leq k$. Then, for any $m \geq 0$ and any $n \geq k+m$, the following diagram is homotopy commutative
\begin{equation}\label{diamante}
         \xymatrix{
&G_{n}X\ar[dl]_{\iota_{k}^{n}} \ar[dr]^{\iota_{m}^{n}} &\\
G_kX   \ar[dr]_{r_{k}} &  &  G_{m}X\\
& X  \ar[ur]_{q_{m}}&  }
 \end{equation}
where $r_k$ is an arbitrary retraction of the $k$-th Ganea cofibration.
\end{theorem}
\begin{proof}
The proof follows by induction.  Since $G_0X = CX$, the result is trivial for $m=0$. Now,  $n-1 \geq k$,
and we apply Proposition \ref{prop:Ganea} to the $(n-1)$-th Ganea cofibration, to the fibration $G_{m+1}X\rTo^{\iota_{m}^{m+1}} G_{m}X\rTo^{j_m} C_m$, to $g = \iota_{m+1}^{n-1}, $ and to  $r= r_k \iota_{k}^{n-1}$  as a retraction of $q_{n-1}$. That is, to  the following diagram:
\begin{equation*}\label{general}
         \xymatrix{
&X \ar[d]_{q_{n-1}}  &&G_{m+1}X\ar[d]^{\iota_{m}^{m+1}} \\
G_nX \ar[r]^{\iota_{n-1}^n} &  G_{n-1}X \ar@/_/[u]_{r_k \iota_{k}^{n-1}}\ar[d]_{ j_{n-1}}  \ar[rru]_{\iota_{m+1}^{n-1} } &&G_mX \ar[d]^{j_m}\\
 & C_{n-1}  && C_m  }
 \end{equation*}
Then, $$\iota_{m}^{m+1} \iota_{m+1}^{n-1} =  \iota_{m}^{m+1} \iota_{m+1}^{n-1} q_{n-1}r_k \iota_{k}^{n-1}
 \Leftrightarrow   \iota_{m+1}^{n-1}\iota_{n-1}^n = \iota_{m+1}^{n-1}   q_{n-1}  r_k \iota_{k}^{n-1} \iota_{n-1}^n, $$
or equivalently
$$\iota_{m}^{n-1}  = q_m r_k i_k^{n-1}  \Leftrightarrow   \iota_{m+1}^{n}= q_{m+1}r_k \iota_k^{n}. $$

\end{proof}

Using this result we can now prove Theorem \ref{th-1}:

\begin{proof}[Theorem \ref{th-1}]
Let $m = \operatorname{sup}_p \{\indcocat X_{(p)}\}$ and let $k = \indcocat X_{(0)}$.  Choose, for every prime $p$, an arbitrary retraction of the $m$-th Ganea cofibration
$$\{\tau_p: G_m (X_{(p)}) \rightarrow X_{(p)}  \mid p  \ \text{ prime}\}$$
and fix $r_k: G_k(X_{0}) \rightarrow X_{(0)}$ a retraction for $X_{(0)}$.
Following Theorem \ref{th:sections} applied to $X_{(0)}$ for $n = k+m$, we have the the following equality
\begin{equation*}
 \iota_{m}^{n}= q_{m}r_k \iota_k^{n},
\end{equation*}
where maps and notations are the same as in (\ref{diamante}). Notice that maps in this case are rational.
Hence,
$$(\tau_p)_{(0)}\iota_{m}^{n}= (\tau_p)_{(0)}q_{m}r_k \iota_k^{n} = r_k \iota_k^{n},$$
since by Theorem \ref{toomer}, $ (\tau_p)_{(0)}q_{m} = Id_{X_{(0)}}$.
So, the rationalisations of the considered retractions  are homotopically equivalent  when composed with $\iota_m^{n}$ where $ n = m+k$.
Then, by Lemma \ref{lemilla}, $\indcocat X \leq m+k$.
\end{proof}

%%%%%%%%%%%%%%%%%%%%%%%%%%%%%%%%
%% Apendice
%%%%%%%%%%%%%%%%%%%%%%%%%%%%%%%%%%%
\section{Appendix: the monodromy action of a fibration}\label{appendix}
The aim of this appendix is to define the monodromy action associated to a fibration in a way it extends the classical monodromy action associated to covering maps, and to dualise the contents of \cite[p.\ 442]{BerHil}. That is, given a fibration $F \rightarrow E \rightarrow B$, and an arbitrary space $X$, we define an action of the group
$[X, \Omega B]$ on the set $[X, F]$, and relate the orbits of this action with the Barrat-Puppe exact sequence associated to the fibration. Although we could not find any reference addressing the contents of this section, we claim no original result in this part.

Let  $ p: E \rightarrow B$ be an arbitrary map. The homotopy fiber of $p$ is the following pullback
$$\xymatrix{
F= PB \times_B E \ar[r]^-{\pi_2}  \ar[d]_{\pi_1} &  E \ar[d]^{p}    \\
PB \ar[r]^{\epsilon_1}&  B } $$
where $F = \{ (w, \ e) \mid w (1) = p (e) \}.$ Let $X$ be an arbitrary space and $\bullet$ denote the adjunction of paths, then there exists an action of the group $[X, \Omega B]$ on $[X, F]$ (the monodromy action)
$$\begin{matrix}
[X, \Omega B]\times [X, F]&\rightarrow & [X, F]\\
(\epsilon,g) & \mapsto & \epsilon \vdash g
\end{matrix}$$ given by
\begin{equation}\label{operation}
(\epsilon \vdash g)(x) = (\epsilon \vdash (g', g'')) (x) = (g'(x) \bullet \epsilon (x), g''(x)),
\end{equation}
where $g'=\pi_1 g$, and $g''=\pi_2 g$.

Notice this action is natural by construction, that is given $f\colon X\rightarrow Y$ then $(\epsilon \vdash g)f=\big( (\epsilon f) \vdash (gf)\big)$.
The following result describes the $[X, \Omega B]$-orbits in $[X, F]$.

\begin{theorem}\label{thoperation} Let $ \xymatrix{F = PB \times_B E \ar[r]^-i  &  E \ar[r]^{p} & B}$ be a fibration and let the Barrat Puppe exact sequence:
 $$ \xymatrix{[X, \Omega B] \ar[r]  & [X, F]   \ar[r]^{i_\ast} & [X, E] \ar[r]^{p_\ast}&  [X, B].}$$
 Then, for $g_1, g_2 \in [X, F]$ we have that $i_\ast (g_1) = i_\ast (g_2)$ if and only if there exists $\epsilon \in [X , \Omega B]$ such that $(\epsilon \vdash  g_1) = g_2$.
\end{theorem}

\begin{proof} Let us write $g_i = (g_i', g_i'')$, for $i= 1, 2$. Since
$i_\ast (g_1) = i_\ast (g_2)$, we have that $g_1'' \simeq_H g_2''$ with $H : X \times I \rightarrow E$.
Define $\epsilon: X \rightarrow  \Omega B$ by
$$\epsilon(x)= g'_1(x)^{-1}\bullet(pH_x)^{-1} \bullet g'_2(x).$$

We now prove that $\epsilon \vdash g_1 $ is homotopic to $g_2$. By definition we have
$$(\epsilon \vdash g_1 ) (x) =  \Big( g'_1(x)  \bullet \big(g'_1(x)^{-1}\bullet(pH_x)^{-1} \bullet g'_2(x)\big) ,  \ g_1''(x)\Big)$$
which is clearly homotopic to the map defined by $$\phi(x) =  \Big( (pH_x)^{-1} \bullet g'_2(x), \ g_1'' (x)\Big).$$ Finally, $\phi$ is homotopic to $g_2$ by the homotopy $\xymatrix{G={(G', G'')}\colon X \times I \ar[r] & F}$ defined by
\begin{align*}
G'(x,s)(t) &= (pH_x)^{-1} ((1-s)t)  \bullet g'_2(x)\\
G''(x,s) & = H_x(s)
\end{align*}
Notice that $G$ is well defined over the fiber, since $p(G''(x,s)) =G'(x,s)(1)$.
\end{proof}

%%%%%%%%%%%%%%%%%%%%%%%%%%%%%%%%%%
% Bibliography
%%%%%%%%%%%%%%%%%%%%%%%%%%%%%%%%%%

\end{document}